\documentclass[a4paper,11pt]{amsart}

\pdfoutput=1

\usepackage[T1]{fontenc}
\usepackage{amssymb,amsfonts,amsmath,amsthm}
\usepackage{url}
\usepackage{color}
\usepackage{enumerate}
\usepackage{epsfig,graphicx,psfrag}
\usepackage{tikz}
\usepackage{asymptote}
\usepackage{pinlabel}

\usepackage[top=1.2in,bottom=1.4in,left=1.4in,right=1.4in]{geometry}

\input{xy}
\xyoption{all}
\objectmargin={3mm}

\newcounter{notes}
\newcommand{\ignore}[1]{}

\newtheorem{theorem}{Theorem}
\newtheorem{proposition}[theorem]{Proposition}
\newtheorem{corollary}[theorem]{Corollary}
\newtheorem{lemma}[theorem]{Lemma}

\newtheorem{observation}[theorem]{Observation}

\theoremstyle{definition}
\newtheorem{definition}[theorem]{Definition}
\newtheorem{remark}[theorem]{Remark}

\newtheoremstyle{theoremwithref}{}{}{\itshape}{}{\bfseries}{.}{.5em}{#1 #2 #3}
\theoremstyle{theoremwithref}

\newcommand{\R}{\mathbb{R}}
\newcommand{\Q}{\mathbb{Q}}
\newcommand{\Z}{\mathbb{Z}}
\newcommand{\N}{\mathbb{N}}

\newcommand{\PSL}{\mathrm{PSL}(2,\mathbb{R})}

\newcommand{\Diff}{\mathrm{Diff}}
\newcommand{\Aff}{\mathrm{Aff}}
\newcommand{\HOS}{\mathrm{Homeo}_+(S^1)}
\newcommand{\HOZ}{\mathrm{Homeo}^\Z_+(\R)}

\newcommand{\id}{\mathrm{id}}
\newcommand{\Fix}{\mathrm{Fix}}
\newcommand{\Supp}{\mathrm{Supp}}

\DeclareMathOperator{\Homeo}{Homeo}

\title{Reconstructing maps out of groups}

\author{Kathryn Mann}
\address{Department of Mathematics, Cornell University, Ithaca, NY 14853, USA
}
\email{k.mann@cornell.edu}

\author{Maxime Wolff}
\address{Sorbonne Universit\'es, UPMC Univ.\ Paris 06, Institut de Math\'ematiques
de Jussieu-Paris Rive Gauche, UMR 7586, CNRS, Univ. Paris Diderot, Sorbonne
Paris Cit\'e, 75005 Paris, France}
\email{maxime.wolff@imj-prg.fr}

\begin{document}

\maketitle
\numberwithin{theorem}{section}

\begin{abstract}
We show that, in many situations, a homeomorphism $f$ of a manifold $M$ may
be recovered from the (marked) isomorphism class of a finitely generated
group of homeomorphisms containing $f$.
As an application, we relate the notions of {\em critical regularity} and
of {\em differentiable rigidity}, give examples of groups of diffeomorphisms
of 1-manifolds with strong differential rigidity, and in so doing
give an independent, short proof of a recent result of Kim and Koberda that
there exist finitely generated groups of $C^\alpha$ diffeomorphisms of a
1-manifold $M$, not embeddable into $\Diff^\beta(M)$ for any $\beta > \alpha > 1$.
\vspace{0.2cm}

\end{abstract}

\section{Introduction}

\subsection{Motivation}
It is a classical and fundamental problem to describe to what extent the
algebraic structure of a group determines the topological spaces
on which the group can act, or constrains the regularity of those actions.
For example, Whittaker \cite{Whittaker} showed that closed topological
manifolds can be completely recovered from the algebraic structure of their groups of
homeomorphisms: an isomorphism between $\Homeo(M)$ and $\Homeo(N)$ implies
that $M = N$ and the isomorphism is an inner automorphism.
This was generalized by Rubin to homeomorphism groups of other topological
spaces, and Filipkiewicz~\cite{Filipkiewicz} improved this to the groups of $C^r$
diffeomorphisms of manifolds, showing that the algebraic structure
of $\Diff^r(M)$ can even detect the regularity $r$.

All of these could be considered {\em recognition} or {\em reconstruction} theorems, 
showing that spaces can be recognized by their transformation groups.
A different approach to the classical problem is to relate the complexity of a
topological space to the algebraic complexity of (finitely generated)
subgroups of its homeomorphism or diffeomorphism groups.
This is, in some sense the ``generalized Zimmer program,'' Zimmer's
conjecture being that groups of high algebraic complexity, namely lattices
of higher rank, cannot act by smooth or volume-preserving diffeomorphisms
on low-dimensional manifolds.

This broad line of investigation has been particularly successful in dimension one.
Here we know several purely algebraic conditions that prevent
finitely generated groups from acting on one-manifolds with a given regularity.
In the $C^0$ setting, this is the presence of left- or circular-orderability.
In class $C^1$,
many obstructions come from the {\em Thurston stability theorem}, while 
in higher regularity this program can be traced back all the way to
Denjoy's work on rotations of the circle.  
To give some more recent examples, Navas~\cite{NavasT} showed that Kazhdan's
property~$T$ is an algebraic obstruction to acting on the circle with
$C^{1+\alpha}$ regularity for $\alpha > 1/2$;
Castro--Jorquera--Navas~\cite{CJN} gave examples of nilpotent groups with
sharp bounds on the H\"older regularity of their actions on the closed interval $I$; and
more recently, Kim and Koberda~\cite{KimKoberda} gave examples of finitely
generated subgroups of ``critical regularity $\alpha$,'' embeddable in
$\Diff^\alpha(M)$ but not in $\Diff^\beta(M)$ for any
$\beta>\alpha$ when $M=S^1$ or~$I$.

\subsection{Results}\label{ssec:Results}
Our aim here is to contribute both to the general program of recognition and reconstruction,
and to the problem of restricting regularity, with a specific
application to the one-dimensional case.

We give general criteria for a group $\Gamma \subset \Homeo(X)$ of
homeomorphisms of a space $X$ to ``reconstruct'' or ``recognize'' other
homeomorphisms of $X$ purely through algebraic relations (Theorem \ref{theo:RecoGene}). 
We also construct groups acting on 1-manifolds with a strong
{\em differentiable rigidity} property (Theorem~\ref{thm:GammaRigid} and following),
by using  recent work of
Bonatti--Monteverde--Navas--Rivas~\cite{BMNR}
and a precise version of the Sternberg linearization theorem.
Building on all this, we deduce the existence of groups with {\em critical
regularity} (Theorem~\ref{coro:CritReg}).
This gives an alternative short proof (and some generalization) of the
critical regularity result of Kim and Koberda mentioned above.
However, their techniques go further in a different direction than ours:
they also give groups whose critical regularity passes to finite index
subgroups, simple groups of given regularity, and define dynamical notions
``$\delta$-fast'' and ``$\lambda$-expansive'' that are useful for explicitly
constructing groups of specified regularity.

The remainder of this introductory section is devoted to giving precise
statements of our results.

\subsection*{First result: map recognition} 
\begin{definition}
  Let $X$ be a topological space, let $\Gamma\subset\Homeo(X)$,
  and let $\mathcal{C}$ be any subset of $\Homeo(X)$.
  We will say that $\Gamma$ {\em recognizes maps in $\mathcal{C}$}
  if for any $f\in\mathcal{C}$, and for any
  $h\in\Homeo(X)$, the existence of a group isomorphism
  \[ \phi\colon\langle\Gamma,f\rangle\to\langle\Gamma,h\rangle \]
  with $\phi_{|\Gamma}=\id_\Gamma$ and $\phi(f)=h$ implies that
  $h=f$.
\end{definition}
Note that if $\Gamma$ recognizes maps in $\mathcal{C}$
and $\Gamma \subset \Gamma'$, then $\Gamma'$ recognizes
maps in $\mathcal{C}$ as well.  Furthermore, the property of recognizing individual maps
in $\mathcal{C}$ is equivalent to the property of recognizing any {\em subset}
of $\mathcal{C}$.

The following theorem, proved in Section \ref{sec:MapRec}, shows that examples
of such groups abound.  We introduce some terminology needed for the statement.
Recall that, for a group $\Gamma \subset \Homeo(X)$ and $\gamma \in \Gamma$,
the {\em support} of $\gamma$ is the closure of the set
$\lbrace x\in X\mid f(x)\neq x \rbrace$.
{\em Non total support} means $\Supp(\gamma) \neq X$.
We say that $\Gamma$ has {\em small supports everywhere} if, for every nonempty
open set $U\subset X$, there exists $\gamma\in\Gamma \smallsetminus \{\id\}$
with $\Supp(\gamma) \subset U$, and that $\Gamma$ {\em has the contraction property}
if, for any nonempty open set $U\subset X$, there exists $\gamma\in\Gamma$
such that $\gamma(X\smallsetminus U)\subset U$.

\begin{theorem}[Map recognition]\label{theo:RecoGene}
  Let $X$ be a Hausdorff topological space, and $\Gamma\subset\Homeo(X)$.
  \begin{enumerate}
  \item If $\Gamma$ has maps with {\em small supports everywhere}, then
    $\Gamma$ recognizes maps in $\Homeo(X)$.
  \item If $\Gamma$ acts on $X$ with the {\em contraction property}, then
    $\Gamma$ recognizes homeomorphisms of $X$ with non total support.
  \end{enumerate}
\end{theorem}
Similar conditions have been used elsewhere in the literature.
The reconstruction theorems of Whittaker, Epstein, and
Rubin~\cite{Epstein,Whittaker,Rubin} all use variations on the idea of small
supports.
To our knowledge, the contraction property was first used (under the more cumbersome name
of ``minimality and strong expansivity'') in the proof by Margulis of the
Tits' alternative in $\Homeo_+(S^1)$; see~\cite{MargulisTits,GhysCircle}.

We also show that Baumslag-Solitar groups give additional examples of groups
with map recognition.
These are needed for our applications and do not fall
in the domain of Theorem~\ref{theo:RecoGene}.
\begin{theorem}\label{theo:BSReco}
  The affine Baumslag-Solitar subgroup $BS(1,n)\subset\Homeo(\R)$
  recognizes maps with compact support.
\end{theorem}
Here $BS(1,n)$ denotes the group generated by the maps
$x\mapsto x+1$ and $x\mapsto nx$.
Theorem~\ref{theo:BSReco}
is proved in Section \ref{sec:BS}, where we actually prove something
stronger -- see Proposition~\ref{prop:BSRecoMieux}.
It would be interesting to find a simple and general condition that would
simultaneously imply both the statements of Theorem~\ref{theo:RecoGene}
and Theorem~\ref{theo:BSReco}.

\subsection*{Application: differential rigidity from critical regularity}
\begin{definition}[Differential rigidity and critical regularity]
  Let $M$ be a manifold and $\alpha\geqslant 1$.
  \begin{enumerate}
  \item A subgroup $\Gamma\subset\Diff^\infty(M)$ is said to be
    {\em $C^\alpha$-rigid} if for all $\beta\geqslant\alpha$, any faithful
    morphism $\Gamma\to\Diff^\beta(M)$ comes from conjugation by some
    element of $\Diff^\beta(M)$.
  \item A subgroup $\Gamma\subset\Diff^\alpha(M)$ is said to have
    {\em critical regularity $\alpha$} if for every $\beta>\alpha$
    there is no faithful morphism
    $\Gamma\to\Diff^\beta(M)$\footnote{In \cite{KimKoberda},
    ``critical regularity $\alpha$'' means something slightly more general:
    it denotes this property of a group,
    and also the property of being
    embeddable into $\Diff^\delta(M)$ for all $\delta < \alpha$ but not
    in $\Diff^\alpha$.}
  \end{enumerate}
\end{definition}
Here $\alpha$ and $\beta$ are assumed to take real values, with the
convention that a map $f\colon M\to M$ is of class $C^\alpha$ if it is
$C^{\lfloor\alpha\rfloor}$ and if it is $\lfloor\alpha\rfloor$-derivatives
are $(\alpha-\lfloor\alpha\rfloor)$-H\"older.
However, most of our work in the 1-dimensional case actually applies to
maps whose regularity is given by more general moduli of continuity.
We assume
H\"older regularity
here only for simplicity of the statement.
See Remark~\ref{rmk:Sternberg-Reg} below.

The following proposition illustrates that {\em critical
regularity follows from differentiable rigidity} in a general sense:
this is the guiding principle and original motivation of our work.
\begin{proposition}[Critical regularity from differential rigidity]\label{prop:DRCR}
  Let $M$ be a manifold and $\alpha\geqslant 1$.
  Let $\Gamma\subset\Diff^\infty(M)$ be $C^\alpha$-rigid, and suppose that
  for some nonempty open set $U$ (possibly equal to $M$),
  $\Gamma$ recognizes maps with support in~$U$. Then
  for any map $f\in\Homeo(M)$ with support in $U$, and any
  $\beta\geqslant\alpha$, the group $\langle\Gamma,f\rangle$ admits a
  faithful morphism to $\Diff^\beta(M)$ if and only if $f\in\Diff^\beta(M)$.
\end{proposition}
The proof is a quick consequence of the definitions, we give it at the
beginning of Section~\ref{sec:DRCR_cercle}.

\subsection*{Examples of groups with differential rigidity and critical
regularity}  
Proposition~\ref{prop:DRCR} motivates the construction of differentiably
rigid groups that have the map recognition property.
We will several examples, described below, when $\dim(M) =1$.
Combined with Proposition \ref{prop:DRCR} and variations on it, these
constructions give a short proof of the following
result, due to Kim and Koberda for $S^1$ and $[0,1]$.
\begin{theorem}[Compare Kim--Koberda~\cite{KimKoberda}]\label{coro:CritReg}
  For $M = S^1$, $\R$, or $[0,1]$, and for all $\alpha > 1$, there exist
  finitely generated subgroups of $\Diff^\infty(M)$ of critical regularity $C^\alpha$.
 \end{theorem}
In fact, as we mentioned above, the statement we obtain
here is valid in much more generality than regularities $C^{1+\alpha}$ with
$\alpha\in\R_+$, but with more general moduli of continuity; we prove
that these finer regularities are detected
by the algebraic structures of the groups.
This implies, in particular, that there exist uncountably many
non-isomorphic
finitely generated groups in each class of critical regularity in
Theorem~\ref{coro:CritReg}, and responds to Question~7.1(1) of~\cite{KimKoberda}.
We do not seek to state Theorem~\ref{coro:CritReg} in maximal degree of
generality here -- see Remark~\ref{rmk:Sternberg-Reg} and the
discussion following.  

The groups we construct for the use
of Proposition~\ref{prop:DRCR} are actually quite easy to describe.
The simplest case is $M =S^1$.
Let $\Gamma_T\subset\PSL$ be a Fuchsian triangle group $(2,3,7)$, with presentation
\[ \Gamma_T = \langle s,r,t \mid  s^2=r^3=t^7=trs=1 \rangle, \]
choose an integer $n\geqslant 2$, consider a proper interval $I \subset S^1$,
and let $\Gamma_A$ be a copy of an affine subgroup containing $BS(1,n)$ and
an extra irrational homothety $x \mapsto \mu x$, acting smoothly on $S^1$
by a conjugate of the affine action on $I$, and by the identity on
$S^1\smallsetminus I$. Let $\Gamma$ denote the group generated by $\Gamma_T$ and $\Gamma_A$.
While there are many choices involved in this construction, we show all resulting groups are rigid: 

\begin{theorem}[Differential rigidity on $S^1$]\label{thm:GammaRigid}
  Any group $\Gamma$ obtained by the construction above is
  $C^{1+\alpha}$-rigid, for all $\alpha>0$.
\end{theorem}

For other $1$-manifolds, the situation is more delicate.
For example, no groups can act in a differentiably rigid way on the closed
interval $I = [0,1]$, as one can always conjugate an action to make it
infinitely tangent to a linear action at $0$,
``double'' the interval at $0$ and glue two copies of the action side by side.
However, our strategy can be adapted to prove critical regularity using
weaker forms of map recognition for group actions on $\R$ and $I$.
On the line, one can arrive at this by lifting maps from $S^1$.
We show:
\begin{proposition} \label{prop:Rigid_lifts}
  Let $\widetilde{\Gamma}$ be the group of lifts to $\Diff(\R)$ of
  elements of a group $\Gamma\subset\Diff^\infty(S^1)$ defined above.
  Then $\widetilde{\Gamma}$ is $C^{1+\alpha}$-rigid, for all $\alpha>0$.
  Moreover, if $\HOZ$ denotes the set of homeomorphisms of $\R$
  which commute with integer translations, then $\widetilde{\Gamma}$
  recognizes maps of $\HOZ$ up to integer translations.
\end{proposition}

On the closed interval, more work is needed.  Let 
$\Gamma_f$ denote the group generated by $\Gamma_A$ as above, acting smoothly
on $I = [0,1]$ and conjugate on $(0,1)$ to the standard affine action,
together with a homeomorphism $f$ with support in $(0,1)$.
For simplicity we suppose also that the set $\{x\in(0,1)\mid f(x)\neq x\}$
is connected. Then we have the following.
\begin{theorem}[Differential rigidity on the interval] \label{thm:IntervalRigid}
  Let $\phi\colon\Gamma_f\to\Diff^{1+\alpha}([0,1])$ be a faithful morphism.
  Then there exists an interval $(a,b)\subset[0,1]$, invariant under
  $\phi(\Gamma_f)$, and a $C^{1+\alpha}$-diffeomorphism
  $h\colon(0,1)\to(a,b)$ conjugating $\phi(\Gamma)$ to the original action on $[0,1]$.
  In particular, this implies $f\in \Diff^{1+\alpha}([0,1])$.
\end{theorem}

\subsection*{Regularity of conjugacies} 
A major ingredient the examples above is a result on {\em regularity of conjugacies}
(Proposition \ref{prop:Takens}), reminiscent of a theorem of Takens, which may
be of independent interest.
We show that the group $\Gamma_A$ described above has the property that, if
$\phi\colon\Gamma_A\to\Diff^{1+\alpha}([0,1])$ is conjugate to the standard
affine action by a homeomorphism $f$, then $f$ is in fact of
class~$C^{1+\alpha}$.
This is the main content of Section \ref{sec:Takens}.

\subsection*{Higher dimension}
We hope that this application to problems of critical regularity
(\textsl{via} Proposition~\ref{prop:DRCR}) provides motivation to construct and
study groups of diffeomorphisms of higher dimensional manifolds with differential
rigidity, or that exhibit the regularity of conjugacies property of
Proposition~\ref{prop:Takens}.
This seems to be a challenging problem, and the situation there may be
quite different. Note, for example, that Harrison~\cite{Harrison, Harrison2}
constructed $C^r$ diffeomorphisms of manifolds
(in all dimensions $\geqslant 2$) that are not topologically conjugate to any
$C^s$ diffeomorphisms for any $s>r$.

\subsection*{Acknowledgements}
The authors thank S. Kim and T. Koberda for feedback on an early version
of this work. K.M. was partially supported by NSF grant DMS 1844516.
This work was started when both authors were in Montevideo, we thank the
Universidad de la Rep\'ublica for its hospitality.


\section{Proof of Theorem~\ref{theo:RecoGene}: map recognition}\label{sec:MapRec}

Statements~(1) and~(2) of this theorem follow the same general strategy of
proof, so we treat them in parallel.
Throughout this section, $X$ denotes a Hausdorff topological space.
We will suppose furthermore that $X$ has cardinal $\geqslant 3$.
Provided that there exists a group acting on it with small supports everywhere
or with the contraction property this implies that $X$ is infinite, and
has no isolated points. In the case $\mathrm{Card}(X)\leqslant 2$,
Theorem~\ref{theo:RecoGene} is immediate.

\begin{lemma}\label{lem:Outil1}
  Let $\Gamma\subset\Homeo(X)$ be a group with maps with small supports
  everywhere. Let $f\colon X\to X$ be a continuous map, and let $x\in X$.
  Then the following holds.
  \begin{enumerate}
  \item If $f(x)\neq x$, then for any sufficiently small neighborhood
    $U_x$ of $x$ and $\gamma\in\Gamma\smallsetminus\{\id\}$ with
    $\Supp(\gamma)\subset U_x$, the maps $\gamma$ and $f$ do not commute.
  \item If $x\not\in\Supp(f)$, then for any sufficiently small neighborhood
    $U_x$ of $x$ and $\gamma\in\Gamma$ with
    $\Supp(\gamma)\subset U_x$, the maps $\gamma$ and $f$ commute.
  \end{enumerate}
\end{lemma}
The proof is a straightforward exercise, which we omit. 
The version for groups with the contraction property is
more interesting:
\begin{lemma}\label{lem:Outil2}
  Let $\Gamma \subset \Homeo(X)$ have the contraction property,
  let $f\colon X\to X$ be a continuous map, and let $x\in X$.
  Then the following holds.
  \begin{enumerate}
  \item If $f(x)\neq x$, then for any sufficiently small neighborhood $U_x$
    of $x$ and every $\gamma\in\Gamma$ mapping $X\smallsetminus U_x$ into $U_x$,
    the maps $f$ and $\gamma f\gamma^{-1}$ do not commute.
  \item If $x\not\in\Supp(f)$, then for any sufficiently small neighborhood
    $U_x$ of $x$, and every $\gamma\in\Gamma$ mapping $X\smallsetminus U_x$ into
    $U_x$, the maps $h$ and $\gamma f\gamma^{-1}$ commute.
  \end{enumerate}
\end{lemma}

\begin{proof}
  The second item is nearly immediate. Simply take $U_x$ contained in
  $X\smallsetminus\Supp(f)$.
  Then the support of $\gamma f\gamma^{-1}$ lies in $U_x$, so
  $f$ and $\gamma f\gamma^{-1}$ have disjoint supports, hence commute.

  For the first item, suppose $f(x)\neq x$.
  Since $X$ has no isolated points,
  there exists some other point $z\in X$ such that the set
  $\{x,f(x),z,f(z)\}$ has cardinality~$4$.
  Let $U_x$ be a neighborhood of $x$ such that $\{z, f(z)\} \cap U_x = \emptyset$ and 
  $U_x  \cap f(U_x) = \emptyset$.
  (Such a neighborhood exists since $X$ is Hausdorff). Let 
  $\gamma\in\Gamma$ satisfy $\gamma(X\smallsetminus U_x) \subset U_x$.
  Note that this also implies that $\gamma^{-1}(X\smallsetminus U_x) \subset U_x$.
  It follows that $f\circ (\gamma f\gamma^{-1})(\gamma z) \in f(U_x)$, but 
  $(\gamma f\gamma^{-1})\circ f (\gamma z) \in U_x$, so $f$ does not commute
  with its conjugate by~$\gamma$.
\end{proof}

The lemma above will allow us to reconstruct maps, first by recovering their support. 

\begin{lemma}\label{lem:OutilSupport}
  Let $\Gamma\subset\Homeo(X)$
  be a subgroup, either with small supports everywhere, or with the
  contracting property. Let $f,h\in\Homeo(M)$ be any homeomorphisms.
  Suppose that there exists a group isomorphism
  \[ \phi\colon\langle\Gamma,f\rangle\to\langle\Gamma,h\rangle \]
  such that $\phi_{|\Gamma}=\id_\Gamma$ and $\phi(f)=h$.
  Then $\Supp(f)=\Supp(h)$.
\end{lemma}

\begin{proof}
  Let $x\in X$ be such that $f(x)\neq x$. Suppose for contradiction
  that $x\not\in\Supp(h)$.
  
  Suppose first that $\Gamma$ has maps with small supports everywhere.
  We use Lemma~\ref{lem:Outil1}. Let $U_x$ be a small neighborhood
  of $x$ and let $\gamma\in\Gamma\smallsetminus\{\id_X\}$ with support
  in $U_x$. Then $f$ does not commute with $\gamma$, while $h$ commutes
  with $\gamma$: this contradicts that~$\phi$ is an isomorphism.
  If instead $\Gamma$ is contracting, we use Lemma~\ref{lem:Outil2}:
  let $U_x$ be a small enough neighborhood of $x$ and let
  $\gamma\in\Gamma$ be an element mapping $X\smallsetminus U_x$ inside $U_x$.
  Then $f$ and $\gamma f \gamma^{-1}$ do not commute,
  while $h$ and $\gamma h \gamma^{-1}$ commute: this again contradicts the
  existence of~$\phi$.
  
  Hence, we have proved the inclusion
  $ \lbrace x\in X\mid f (x)\neq x \rbrace \subset \Supp(h). $
  Taking closures, this implies $\Supp(f)\subset\Supp(h)$.
  The reverse inclusion follows since the 
  roles of $f$ and $h$ are symmetric.
\end{proof}

We now conclude the proof of Theorem~\ref{theo:RecoGene}.
Let $X$ be a Hausdorff topological space, and, as a first case, assume
$\Gamma\subset\Homeo(X)$ is a group with small supports everywhere.
Let $f,h\in\Homeo(X)$ be any two homeomorphisms, and suppose there
exists a group isomorphism $\phi$ as in the statement of Theorem~\ref{theo:RecoGene}.
Suppose for contradiction that there exists $x\in X$ such that
$h(x)\not\in\{x,f(x)\}$. Let $U_x$ be a neighborhood of $x$ such that
$h(x)$ is not in $U_x\cup f(U_x)$, and let $\gamma\neq\id_X$ be an element
of $\Gamma$ with support in $U_x$; suppose furthermore (without loss of
generality) that $x\in\Supp(\gamma)$. Then the commutator map
$[f^{-1},\gamma]=f\circ\gamma^{-1}\circ f^{-1}\circ\gamma$ has support in
$U_x\cup f(U_x)$, while $h(x)\in\Supp([h^{-1},\gamma])$. Hence,
the maps $f_2=[f^{-1},\gamma]$ and $h_2=[h^{-1},\gamma]$ have distinct
supports. On the other hand, the isomorphism $\phi$ restricts to an
isomorphism $\phi_2\colon\langle\Gamma,f_2\rangle\to\langle\Gamma,h_2\rangle$
with the same properties, and this gives a contradiction with
Lemma~\ref{lem:OutilSupport}. Thus, for every $x$ we have $h(x)\in\{x,f(x)\}$,
and symmetrically we have $f(x)\in\{x,h(x)\}$. This implies $h=f$.

Finally, suppose instead that $\Gamma$ is contracting, let $f\in\Homeo(X)$ be a
map with non total support, let $h$ be any map, and suppose there exists
a group isomorphism $\phi$ as above. By Lemma~\ref{lem:OutilSupport}
we have $\Supp(f) = \Supp(h)$, in particular $h$ also has non total support
and $f$ and $h$ again play symmetric roles. For contradiction suppose there
exists $x\in X$ such that $h(x)\not\in\{x,f(x)\}$. Let $y\in X$ be a point
which is not in $\Supp(f)$, and
distinct from $x,f(x),h(x)$.
As $X$ is Hausdorff, there exist a neighborhood $U_x$ of $x$, and a neighborhood
$U_y$ of $y$, such that the sets
$h(U_x)$, $U_x\cup f(U_x)$ and $U_y$ are pairwise disjoint,
and such that $U_y\cap\Supp(f)=\emptyset$.
Let $\gamma_x,\gamma_y\in\Gamma$ be such that $\gamma_x(X\smallsetminus U_x)\subset U_x$
and $\gamma_y(X\smallsetminus U_y)\subset U_y$.
The map $\gamma_y^{-1}\circ f\circ\gamma_y$ has support in $U_y$, so the
map $f_3 = \gamma_x^{-1}\gamma_y^{-1}f\gamma_y\gamma_x$ has support in
$U_x$. It follows that the map $f_4 = [f^{-1},f_3]$ has support in
$U_x\cup f(U_x)$, while the map $h_4 = [h^{-1},h_3]$ has support in
$U_x\cup h(U_x)$. Also, $h_4$ is not the identity in $h(U_x)$, simply
because $h$ is non trivial. However, $\phi$ restricts to a group
isomorphism $\langle\Gamma,f_4\rangle\to\langle\Gamma,h_4\rangle$,
thus Lemma~\ref{lem:OutilSupport} yields a contradiction, as in the
preceding case.
\qed \\

\section{Proof of Theorem \ref{theo:BSReco}: maps recognized by $BS(1,n)$}  \label{sec:BS}
For $2 \leqslant n \in \N $, let 
$BS(1,n)=\langle a,b\mid aba^{-1}=b^n\rangle$ denote the Baumslag--Solitar group, 
with its standard affine action on the real line, defined by
$a\colon x\mapsto nx$ and $b\colon x\mapsto x+1$.

This section is devoted to the proof of the following stronger version of
Theorem~\ref{theo:BSReco}. Here we require only a morphism, not an isomorphism
between groups.
We will need to use this weaker hypothesis in our discussion of an analog
of differential rigidity on the closed interval in Section \ref{sec:CritRegInterval}.
Theorem~\ref{theo:BSReco} follows immediately from the statement below by
taking $\phi$ to be an isomorphism, and applying the result also to $\phi^{-1}$.
\begin{proposition}\label{prop:BSRecoMieux}
  Let $f,h\in\Homeo_+(\R)$, and suppose $f$ has compact support.
  Suppose there is a (not necessarily injective) morphism
  $\phi: \langle BS(1,n),f\rangle\to \langle BS(1,n),h\rangle$ restricting
  to the identity on $BS(1,n)$ and mapping $f$ to $h$. Then either
  $h$ is a translation, or we have, for all $x\in\R$, $h(x)\in\{x,f(x)\}$.
\end{proposition}

In the statement, $\langle BS(1,n),f\rangle \subset \Homeo(\R)$ denotes the
group generated by $BS(1,n)$ and $f$;  the same applies to $h$.
The proof of Proposition~\ref{prop:BSRecoMieux}, like that of
Theorem~\ref{theo:RecoGene}, is through a careful study of the supports of
(non)-commuting elements. The main technical tool is the following.

\begin{proposition}\label{prop:ToolBS}
  Let $h\in\Homeo_+(\R)$ be different from a translation. Then the following
  statements are equivalent:
  \begin{enumerate}
  \item $h$ has compact support.
  \item There exists a compact set $K\subset\R$ such that
    for every element $u\in BS(1,n)$
    with $u(K)\cap K =\emptyset$, we have $[h,uhu^{-1}]=1$;
    and for every element $\gamma\in BS(1,n)$, the commutator $[h,\gamma]$
    also satisfies the same hypothesis: there exists a compact $K'$ such that
    for all $u\in BS(1,n)$ with $u(K')\cap K'=\emptyset$,
    we have $[[h,\gamma], u[h,\gamma]u^{-1}]=1$.
  \end{enumerate}
\end{proposition}
In the second point of the statement above the set $K'$ depends on $\gamma$.

Before embarking on the proof,
we begin with an easy and useful lemma.
If $u,v\in\Homeo_+(\R)$ we write $u<v$ if for all $x\in\R$, $u(x)<v(x)$.
We write $u\leqslant v$ if for all $x\in\R$, $u(x)\leqslant v(x)$.
\begin{lemma}\label{lem:PtFixe}
  Let $u,v\in\Homeo_+(\R)$ be commuting maps, each without fixed points.
  Suppose $u<v$ and there exists $w\in\Homeo_+(\R)$ with $w u w^{-1}=v$.
  Then $w$ has fixed points.
\end{lemma}
\begin{proof}
  Up to replacing $u$, $v$ and $w$ with their inverses and
  switching the role of $u$ and $v$, we may assume
  without loss of generality that $\id<u<v$.
  Since $v$ is fixed point free, up to conjugacy we may further assume $v(x) = x+1$.
  Since $u$ commutes with $v$, it has compact fundamental domain $[0,1]$, hence
  there exists $\varepsilon>0$ such that for all $x\in\R$, we have
  $u(x)\leqslant x+1-\varepsilon$. By a simple induction this implies that,
  for any integer $N\geqslant 1$, we have
  $u^N(x)\leqslant x+N-N\varepsilon$ and $u^{-N}(x)\geqslant x-N+N\varepsilon$.
  
  Now let $x\in\R$ be any point. If $w(x)<x$ then take $N\geqslant 1$
  such that $w(x)-x+N\varepsilon>0$, and set $y=u^N(x)$. Then we have
  \[w(y) = v^N w(x) = w(x) + N \geqslant w(x) - x + N\varepsilon + u^N(x) > y. \]
  Hence, $w(x)<x$ and $w(y)>y$: this implies that $w$ has
  at least one fixed point.
  If instead $w(x) > x$, take $N\geqslant 1$ with $x - w(x) - N\varepsilon >0$;
  the same reasoning shows that $w(u^{-N})(x) < u^{-N}(x)$, also implying $w$
  has a fixed point.
\end{proof}

Going forward, we denote by $A$ the abelian subgroup $\Z[1/n] \subset BS(1,n)$
consisting of translations, \textsl{i.e.}, the normal subgroup generated
by $b$.
For $t \in \R$, let $\tau_t\colon\R\to\R$ denote
the translation $x \mapsto x+ t$.
We note the following standard fact.

\begin{observation}\label{obs:CentralTrans}
  The centralizer of $A$ in $\Homeo_+(\R)$ is the translation subgroup.
  Indeed, this is true when $A$ is replaced with any dense subgroup of translations.
\end{observation} 

\begin{proof}[Proof of Proposition~\ref{prop:ToolBS}]
  The implication $(1) \Rightarrow (2)$ is immediate,
  simply take $K=\Supp(h)$, and, for any $\gamma\in\Gamma$,
  take $K'=\Supp([h,\gamma])$.
  So we need only prove the converse.
  The proof has three preliminary steps.
  We state these as Lemmas since we will later apply them to
  a commutator involving
  $h$, rather than $h$.

  \begin{lemma}\label{lem:Obs1}
    Suppose $h \in \Homeo_+(\R)$, and $K \subset \R$ is a compact set such that
    $[h,uhu^{-1}]=1$ for each $u$ where $u(K)\cap K =\emptyset$.
    If the germ of $h$ at either
    $+\infty$ or $-\infty$ is a nontrivial translation, then $h$ is a
    translation.
  \end{lemma}
  \begin{proof}
    Suppose that the germ at $+\infty$ agrees with that of translation
    by some real number $t_0\neq 0$.
    (The case for the germ at $-\infty$ is exactly the same.)
    By Observation~\ref{obs:CentralTrans}, to show that $h$ is a translation
    it suffices to show that $h$ commutes with arbitrarily small translations.
   
    Let $x\in\R$ and $k\geqslant 0$. By hypothesis,
    $h$ commutes with $a^{-k}b^{-N}hb^Na^k$ provided $N$ is large enough.
    Also, for $N>0$ large enough,
    we have
    \[ a^{-k}b^{-N}hb^Na^k(x)=x+\frac{t_0}{n^k} \]
    and
    \[ a^{-k}b^{-N}hb^Na^k(h(x))=h(x)+\frac{t_0}{n^k}. \]
    This yields $h(x)+\frac{t_0}{k}=h(x+\frac{t_0}{k})$ and we are done.
  \end{proof}
  \begin{lemma}\label{lem:Obs2}
    Suppose $h\in \Homeo_+(\R)$ is as in the previous lemma, namely, there
    is a compact $K \subset \R$ such that 
    $[h,uhu^{-1}]=1$ for each $u$ where $u(K)\cap K =\emptyset$.
    Suppose also that $h$ does not have compact support. Then $\Fix(h) = \emptyset$.
   \end{lemma}
  \begin{proof}
    From Lemma~\ref{lem:Obs1} above, either $h$ is a translation (in which case we are done) or
    one germ of $h$, without loss of generality say at $+\infty$,
    is non trivial and not equal to that of a translation.
    Equivalently, the displacement map $x\mapsto h(x)-x$ is not constant
    in any neighborhood of $+\infty$.
    
    Suppose for contraction that $h$ does have a fixed point, $x_0\in\R$.  
    We will show that $h$ has a dense subset of fixed points, \textsl{i.e.}
    $h = \id$, contradicting that $h$ was assumed to have noncompact support.
    
    Let $x_1\in\R$ and $\varepsilon>0$. We want to prove that $h$ has a
    fixed point in the interval $(x_1-\varepsilon,x_1+\varepsilon)$.
    Let $K$ be the compact set given by condition $(2)$, and let $C >0$
    be such that $K \subset (-C, C)$.
    Set $m=\max(x_0,x_1)$, and let $y_1$ be a point in $(m+2C,+\infty)$
    where the displacement of $h$ is not locally constant. Then we can find
    a point $y_2\in(m+2C,+\infty)$ such that the displacements
    $v_1=h(y_1)-y_1$ and $v_2=h(y_2)-y_2$ are independent over $\Q$.
    Now we make the following claim.
    
    \noindent {\bf Claim.} Let $u,v\in(-\infty, m)$ differ by $v_1$ or $v_2$.
    Then $u$ is a fixed point of $h$ if and only if $v$ is.
    
    Let us prove this claim.  We treat the case where $v=u+v_1$ and
    $u$ is a fixed point of $h$, the other cases are symmetric. Let $\delta>0$.
    Since $A=\Z[1/n]$ acts minimally on $\R$, and since $h$ is continuous
    at $y_1$, we can find a point $y_1'\in A\cdot u$ such that $|y_1-y_1'|<\delta$,
    and such that $|(h(y_1)-y_1)-(h(y_1')-y_1')|<\delta$. Provided $\delta$
    is small enough, we also have $|u-y_1'|>2C$. Let $t=u-y_1'\in A$. By
    hypothesis, $h$ commutes with $\tau_t h\tau_t^{-1}$, and since $u$ is
    a fixed point of $h$, this implies that $\tau_t h\tau_t^{-1}(u)$ is also
    a fixed point of~$h$.
    Now, $\tau_t h\tau_t^{-1}(u)=u+(h(y_1')-y_1')$ is within distance $\delta$ from
    $v$: hence $v$ admits fixed points of $h$ in all its neighborhoods,
    and the claim is proved.
    
    Now we can finish the proof of the lemma. Since $v_1$ and $v_2$ are
    independent over $\Q$, there exist $p,q\in\Z$ such that
    $|(x_1-x_0)-(pv_1+qv_2)|<\varepsilon$. 
    Taking $p$ and $q$ to be large, we can also suppose that
    the vectors $pv_1$ and $qv_2$
    have opposite sign. So up to exchanging the two, suppose $pv_1<0$.
    The claim above implies that $x_0+v_1 \in \Fix(h) \cap (-\infty, m)$ and hence,
    can be applied iteratively, showing that
    $x_0+pv_1 \in \Fix(h) \cap(-\infty, m)$. A similar inductive argument with
    $x_0+pv_1$ playing the role of $x_0$ shows that
    $x_0+pv_1 + qv_2 \in \Fix(h)$. This proves the lemma.
  \end{proof}
  \begin{lemma}\label{lem:Obs3}
    Suppose $\tau_{t_0}\in A$ is such that $[h,\tau_{t_0}h\tau_{t_0}^{-1}]=1$
    and $[h,\tau_{t_0}]\neq 1$. Then $[h,\tau_{t_0}]$ has compact support.
  \end{lemma}
  \begin{proof}
    For contradiction, suppose that the hypotheses of the lemma hold, and
    suppose as well that the map $g=[h,\tau_{t_0}]$ is not a translation and
    does not have compact support.
    Then we can apply Lemmas~\ref{lem:Obs1} and~\ref{lem:Obs2} to $g$, and
    deduce that $g$ has no fixed points in $\R$.
    Hence, we have $h>\tau_{t_0}h\tau_{t_0}^{-1}$
    or $h<\tau_{t_0}h\tau_{t_0}^{-1}$. In either case, Lemma~\ref{lem:PtFixe}
    immediately gives a contradiction: hence, $g$ has compact support or is
    a translation. But the latter case is forbiden, by Lemma~\ref{lem:PtFixe}
    (and Lemma~\ref{lem:Obs2} applied to $h$).
    Hence $g$ has compact support.
  \end{proof}
  Now we can finish the proof of Proposition~\ref{prop:ToolBS}.
  Suppose $h \in \Homeo_+(\R)$ satisfies (2) and is not a translation.
  By Observation \ref{obs:CentralTrans} the set of $t\in\R$ such
  that $h$ commutes with $\tau_t$ is nowhere dense, hence the set
  $A_0=\{\tau_t\in A\mid [h,\tau_t]\neq 1\}$ is dense in the set of
  translations. Also, for $t$ large enough
  and $\tau_t\in A_0$, the maps $h$ and $\tau_t h\tau_t^{-1}$ commute, hence,
  by Lemma~\ref{lem:Obs3}, both germs of $[h,\tau_t]$ are trivial.
  Thus, both germs of $h$ are $t$-periodic, for a set of real numbers $t$
  which has accumulation points.
  This implies that both germs of $h$ have constant displacement, and by
  Lemma~\ref{lem:Obs1}
  this displacement is zero. Hence, $h$ has compact support.
\end{proof}
Using this, we prove the main result of this section.  

\begin{proof}[Proof of Proposition~\ref{prop:BSRecoMieux}]
  Let $f \in \Homeo_+(\R)$ have compact support, and suppose that
  $\phi: \langle BS(1,n),f\rangle\to \langle BS(1,n),h\rangle$ is a morphism restricting to 
  the identity on $BS(1,n)$, and with $\phi(f) = h$.
  Suppose also that $h$ is not a translation.
  Using Proposition~\ref{prop:ToolBS} and commutation relations among $f$ and
  elements of $BS(1,n)$ we can conclude that $h$ has compact support.
  As in the proof of Theorem \ref{theo:RecoGene}, we will now show that
  for all $x\in\R$, $h(x)\in\{x,f(x)\}$.

  Suppose for contradiction that there is some point $x$ with $h(x)\not\in\{x,f(x)\}$.
  Let $U$ be a small neighborhood of $x$, chosen small enough so that $U$,
  $f(U)$ and $h(U)$ are pairwise disjoint, and so that no image of $U$
  under translation intersects both $h(U)$ and $f(U)$ simultaneously.  
  Let $\gamma\in BS(1,n)$ be a dilatation
  with fixed point in $U$, and with derivative large enough so that
  $\gamma^{-1} (\Supp(h) \cup \Supp(f)) \subset U$.  
  
  Then the support of the map  
  \[ f_2=\gamma^{-1} f\gamma\circ f\gamma^{-1} f\gamma f^{-1} \]
  is contained in $U \cup f(U)$ and, 
  similarly, the support of  
  \[ h_2=\gamma^{-1} h\gamma\circ h\gamma^{-1} h\gamma h^{-1} \]
  is contained in $U \cup h(U)$.  
  Let $y$ be the rightmost point of $\Supp(h_2) \cap U$.  
  
  Since the action of $A$ is minimal, we can find $\tau \in A$ such that
  $\tau y \in (h(U) - Fix(h_2))$.  In particular, this means that
  $[h_2,\tau^{-1}h_2\tau] \neq 1$, since $\Supp(\tau^{-1}h_2\tau)$ is not
  a $h_2$-invariant set. However, our choice of $U$ ensures that
  $f_2$ and  $\tau^{-1}f_2\tau$ will have disjoint support, and therefore
  commute. This gives the desired contradiction.
\end{proof}


\section{Regularity of conjugacies}  \label{sec:Takens}

One of our ingredients for differential rigidity will be the following
analogue of a theorem of Takens~\cite{Takens}. Takens' theorem states that a
homeomorphism between two smooth manifolds $M$ and $N$, which conjugates
$\Diff^r(M)$ to $\Diff^r(N)$, is necessarily a diffeomorphism of class $C^r$.
Here we specialize to $M=N=[0,1]$, but need only a conjugacy between a
finitely generated affine subgroup.

As in the introduction, let $n\geqslant 2$ and consider the Baumslag-Solitar
group $BS(1,n)$, with
its affine action, together with an extra homothety $x\mapsto \mu x$, with
$\mu \notin \Q$, and let $\Gamma_A$ denote this subgroup of $\Aff(\R)$.
There is a conjugate of this action to an action by diffeomorphisms on $(0,1)$;
which may even taken to be $C^\infty$-tangent to the identity
at~$0$ and~$1$.
\begin{proposition}\label{prop:Takens}
  Let $\alpha>0$, and let $\phi\colon\Gamma_A\to\Diff^{1+\alpha}([0,1])$ be
  an action, $C^0$-conjugate to the standard affine action, so there exists
  a homeomorphism $f\colon(0,1)\to\R$
  such that for every
  $\gamma\in\Gamma_A$, we have $f\circ\phi(\gamma)=\gamma\circ f$.
  Then $f$ is of class~$C^{1+\alpha}$.
  The same holds if $1+\alpha$ is replaced with any modulus of continuity
  $r + \omega$ satisfying Sternberg linearization, as discussed below.
\end{proposition}
The proof has two main ingredients.  The first is a recent result of 
Bonatti--Monteverde--Navas--Rivas~\cite{BMNR}.

\begin{theorem}[Theorem 1.3 and 1.7 in~\cite{BMNR}]\label{thm:BMNR}
  If $BS(1,n)$ acts by
  $C^1$ diffeomorphisms of $[0,1]$ with no fixed point in $(0,1)$ and
  non-Abelian image, then the action
  is $C^0$ conjugate to the standard action, and the the derivative of $a$
  at its (unique) interior fixed point is $\pm n$.
\end{theorem}
The second ingredient is the Sternberg linearization theorem, or more precisley,
Yoccoz's proof of this theorem in \cite{Yoccoz}, which applies 
to a more general setting than $C^{1+\alpha}$ regularity.
Using this Proposition~\ref{prop:Takens} can be seen to hold when
$C^{1+\alpha}$ is replaced by any modulus of continuity to which
this proof applies.
We now describe the context of interest to us.

Recall that if $\omega\colon[0,+\infty)\to[0,+\infty)$ is a homeomorphism, 
a map $f\colon\R\to\R$ is said to be {\em $\omega$-continuous} if for some
$C>0$ we have
$|f(x)-f(y)|\leqslant C\omega(|x-y|)$ for all $x,y\in\R$. For $\omega(t)=t$, or
$\omega(t)=t^\alpha$, this is the notion of Lipschitz, or H\"older functions,
respectively.
A map $f$ is said to be of class $C^{r+\alpha}$ if it is $C^r$ and $f^{(r)}$ is
$\omega$-continuous.
\begin{theorem}[Sternberg linearization]\label{thm:Sternberg}
  Let $\omega\colon[0,+\infty)\to[0,+\infty)$ be a homeomorphism, and suppose
  that there exists an increasing map $\nu\colon[0,+\infty)\to[0,+\infty)$, 
  which sends $(0,1)$ into $(0,1)$, such that for all $x\in[0,1]$ and
  $t\in[0,+\infty)$ 
  we have $\omega(tx)\leqslant \nu(t)\omega(x)$.
  Let $f$ be a germ of a diffeomorphism of $\R$, with $f(0)=0$ and $f'(0)=a<1$,
  of class $C^{r+\omega}$, with $r\geqslant 1$, or of class $C^r$, $r\geqslant 2$.
  Then there exists a unique germ $h$ of diffeomorphism of $\R$, with
  $h(0)=0$ and $h'(0)=1$, with same regularity as $f$, and such that
  $h$ conjugates $f$ into the multiplication by $a$.
\end{theorem}
\begin{remark}\label{rmk:Sternberg-Reg}
  The condition on $\omega$ of existence of such a map $\nu$ is sufficient
  to make $C^{r+\omega}$, stable under composition for all $r\geqslant 1$.  
  The condition $\nu(0,1)\subset(0,1)$
  comes into play in the proof
  in regularity $C^{1+\omega}$.
  Examples for $\omega$ include maps equal to $x\mapsto x^\alpha\ln(1/x)^\beta$
  for small $x$, for $\alpha\in(0,1)$ and $\beta\in\R$.
  In one dimensional dynamics some phenomena depend in a
  very subtle way on the regularity; see~\cite[Paragraph~4.1.4]{Navas} for
  a great variation of examples.
\end{remark}

Note that, by choosing appropriate regularities, \textsl{e.g.} using
a function agreeing near $0$ with
$x\mapsto x^\alpha\ln(1/x)^\beta$ and varying $\alpha$ and $\beta$, this shows 
there exist uncountably many finitely generated groups with critical
regularity~$\alpha$, even in the broader sense of the definition of critical
regularity mentioned in the footnote in Section~\ref{ssec:Results}.

We do not give the proof here, as it is classical, but refer the reader
to the proof appearing in
Yoccoz~\cite[Appendice~4]{Yoccoz}.
(See also Navas~\cite[Theorem~3.6.2]{Navas}.)
While our statement of Theorem~\ref{thm:Sternberg} is
more general than that of Yoccoz,
his proof works in this setting as well:  one applies the
Picard-Banach fixed point
theorem to an operator on a Banach space of functions of a given regularity,
a fixed point of this operator gives the map $h$.  It follows
that there is no loss of regularity between the map $f$ and the conjugating
map $h$, contrarily to Sternberg's original proof.
The condition on $\nu$ in our statement is easily verified to be a sufficient condition for 
the operator used in the proof to be a contraction when $r=1$.  (However, the reader should
keep in mind that Theorem~\ref{thm:Sternberg} is false in regularity $C^1$;
a counterexample was given by Sternberg himself.)
S. Kim and Th. Koberda inform us that this condition has a natural equivalent
formulation, called {\em sub-tameness} of $\omega$ in~\cite{CKK}.
It is also shown in~\cite{CKK} that one may equally well work only with
{\em concave} moduli of continuity; however we find the $\nu$ condition
most straightforward to use in Yoccoz's proof.

Now we give the proof of Lemma~\ref{prop:Takens}.
\begin{proof}[Proof of Lemma~\ref{prop:Takens}]
  We assume for simplicity that $f$
  is orientation preserving, this does not affect the argument.
  From Theorem~\ref{thm:BMNR}, $\phi(a)$ has derivative $n$ at $h(0)$.
  By Sternberg linearization theorem, there exists a unique germ $[h]$ of
  $C^{1+\alpha}$-diffeomorphism at $0\in\R$, conjugating $\phi(a)$ to
  multiplication by $n$, and such that $f'(0)=1$. In other words,
  there exists a neighborhood $(-\delta,\delta)$ and a map
  $h\colon(-\delta,\delta)\to\R$ sending $0$ to $\varphi(0)$,
  such that $h'(0)=1$, and $\phi(a)(h(x))=h(nx)$ for all $x$ small enough.
  Note that the map $x\mapsto \phi(\mu)(h(\frac{x}{\phi(\mu)'(0)}))$
  satisfies the same conditions, hence defines the same germ at $0$,
  by uniqueness. Thus, $h$ conjugates $\phi(a)$ to multiplication by $n$, and,
  simultaneously, conjugates $\phi(\mu)$ to multiplication by some scalar
  $\phi(\mu)'(0)$. Considering the action of $\phi(\mu)$ on the translation
  subgroup of $\phi(\Gamma_A)$,
  we conclude that $\phi(\mu)'(0)=\mu$. 
 
  Hence, the map $h\circ f^{-1}$,
  which is defined on $(-\delta,\delta)$, commutes with a dense group of
  dilatations and so is itself a multiplication
  by a scalar.  In particular, $f$ is of class $C^{1+\alpha}$
  on some neighborhood $U$ of $0$.  
  This is enough to deduce that $f$ has $C^{1+\alpha}$ regularity everywhere,
  since for any compact set $K$, there exists $\gamma \in \Gamma_A$  with
  $\gamma(K) \subset U$ and $\phi(\gamma)(f(K)) \subset f(U)$, allowing us
  to write $f$ on $K$ as a composition of locally $C^{1+\alpha}$ maps.
\end{proof}


\section{Differential rigidity and critical regularity} \label{sec:DRCR_cercle}

This and the following section are devoted to giving examples of groups with
differential rigidity and critical regularity.
Our guiding principle is Proposition \ref{prop:DRCR}, which we prove now.

\begin{proof}[Proof of Proposition \ref{prop:DRCR}]
  If $f\in\Diff^\beta(M)$ then of course, the inclusion maps the group
  $\langle\Gamma,f\rangle$ into $\Diff^\beta(M)$. Conversely, suppose that
  $\varphi\colon\langle\Gamma,f\rangle\to\Diff^\beta(M)$
  is a faithful morphism. Since $\beta\geqslant\alpha$, the restriction of $\varphi$
  to $\Gamma$ coincides with the conjugation by some element $g^{-1}\in\Diff^\beta(M)$;
  denote by $c_g$ the inverse of this conjugation. Hence
  $c_g\circ\varphi\colon\langle\Gamma,f\rangle\to\Diff^\beta(M)$ is a faithful
  morphism, restricting to the identity on $\Gamma$ and mapping $f$ to
  $c_g(\varphi(f))$.
  By recognition and since $f$ has support in $U$, it follows that $f=c_g(\varphi(f))$,
  hence $f$ is of class $C^\beta$.
\end{proof}

\subsection{Proof of Theorem~\ref{thm:GammaRigid}}
The remainder of this section is devoted to the proof of Theorem~\ref{thm:GammaRigid}, 
describing examples of $C^{1+\alpha}$-rigid
groups of diffeomorphisms of $S^1$, for all $\alpha>0$.
We note that examples of $C^3$-rigid such groups have actually been known
for some time: the notion of differential rigidity
essentially appeared in work of Ghys~\cite{GhysRigide}, where he proved
that representations of surface groups with maximal Euler class into
$\Diff^r(S^1)$, $r\geqslant 3$, are $C^r$-conjugate to representations
in $\PSL$. Together with an observation of Calegari~\cite{CalegariForcing},
this implies that, for example the Fuchsian $(2,3,7)$--triangle group in
$\PSL\subset\Diff^\infty(S^1)$ is $C^3$-differentiably rigid.
For the proof of Theorem~\ref{thm:GammaRigid},
it will be convenient to work with the following consequence
(essentially a restatement) of the theorem of
Bonatti, Monteverde, Navas and Rivas given above at Theorem~\ref{thm:BMNR}.

\begin{corollary}\label{cor:BMNR}
  Let $\phi\colon BS(1,n)\to\Diff^1([0,1])$ be a faithful morphism.
  Then there exists an integer $m\geqslant 1$, and $m$ open intervals
  $I_1,\ldots,I_m\subset[0,1]$, each invariant under the action by $\phi$,
  and on which the $\phi$-action of $BS(1,n)$ is $C^0$-conjugate (possibly
  by an orientation reversing homeomorphism) to the standard action of
  $BS(1,n)$ on $\R$. Moreover, $\phi(b)$ restricts to the identity on
  $[0,1]\smallsetminus \cup_j I_j$, and $\phi(a)$ has derivative $\pm n$
  at its (unique) fixed point in each $I_j$.
\end{corollary}
\begin{proof}
  Let $I_1, I_2, \ldots$ be the connected components of
  $[0,1] \smallsetminus \Fix(\phi(BS(1,n)))$.  Apply
  Theorem~\ref{thm:BMNR} to the restriction of the action to each $I_i$.
  If the action is faithful on some $I_i$, then $\phi(a)$ has derivative $\pm n$;
  since $\phi(a)$ is $C^1$ there can only be finitely many such.
  It remains only to show that every non-faithful action of $BS(1,n)$ on
  the line has $b$ in its kernel.  This easily follows from the
  observation that every nontrivial element of $BS(1,n)$ has a normal form 
  $a^{-i}b^ja^k$ for some $i,k\geqslant 0$ and $j\in\Z$.  Thus, 
  the only nontrivial, proper, torsion free quotient
  of $BS(1,n)$ is $\Z\simeq BS(1,n)/\langle b\rangle$.
\end{proof}

\begin{proof}[Proof of Theorem~\ref{thm:GammaRigid}]
  Let $\alpha>0$ and let $\phi\colon\Gamma\to\Diff_+^{1+\alpha}$ be a
  faithful morphism.
  
  The bulk of the proof is devoted to showing that the action of $\phi(\Gamma_T)$
  is minimal, which we do now.
  Calegari~\cite{CalegariForcing} showed that, for any nontrivial action of $\Gamma_T$
  on $S^1$ by homeomorphisms, the Euler number of the action must be maximal.
  It then follows from work of Matsumoto \cite{Matsumoto} that the action is
  {\em semi-conjugate} to the standard one.
  Supposing, for contradiction, 
  that $\phi(\Gamma_T)$ is not minimal, this means that
  there is an
  invariant closed set $K \subset S^1$, homeomorphic to a Cantor set, and
  a surjective, monotone, degree one map $s: S^1 \to S^1$ collapsing each
  complimentary region of $K$ to a point, which intertwines the action of
  $\phi(\Gamma_T)$ with the standard action.
  
  Let $\mathring{K}$ denote the set of two-sided accumulation points of $K$.
  This is also a $\phi(\Gamma)$--invariant set, and the restriction of $s$ to $\mathring{K}$ is a 
  homeomorphism conjugating $\phi(\Gamma_T)$ to 
  the standard action of $\Gamma_T$ on $s(\mathring{K})$, which is a dense subset of $S^1$. 
  Thus, the action of $\phi(\Gamma_T)$ on $\mathring{K}$ has the contraction property.
  Adapting Lemma~\ref{lem:Outil2} to this setting, the relations in $\Gamma$
  imply that the following hold for all $x \in \mathring{K}$:
  \begin{enumerate}
  \item if $s(x)\not\in\Supp(b)$, then $\phi(b)(x)=x$ and $\phi(a)(x)=x$, and 
  \item if $s(x)\not\in\Fix(b)$, then $x\in\Supp(\phi(b))$.
  \end{enumerate}
  Here as in Section~\ref{sec:BS} we denote by $a\colon x\mapsto nx$
  and $b\colon x\mapsto x+1$ the two standard generators of $BS(1,n)\simeq\Gamma_A$.
  From (1), we know that $\phi(a)$ and $\phi(b)$ share a fixed point
  in $S^1$, so we may regard them as acting on the interval.
  Corollary~\ref{cor:BMNR} then asserts that
  there is a finite collection of open intervals $I_1,\ldots, I_m \subset S^1$
  on which the action of $\langle a,b\rangle$
  is topologically conjugate to the standard action, with
  $\Fix(\phi(b))= S^1 \smallsetminus \cup_j I_j$. 
  Since $s(\mathring{K})$ is dense, its intersection with $\Fix(b)$ has
  infinite cardinality, so (1) implies that the compliment of $\cup_j I_j$
  has infinite cardinality. In particular, this complement
  contains some open interval $J$ which intersects $\mathring{K}$.
  Similarly, (2) implies that some nonempty subcollection of the intervals $I_1, \ldots, I_m$
  have nontrivial intersection with $\mathring{K}$.
  Reindexing if needed, we suppose $I_1 \cap \mathring{K} \neq \emptyset$.
   
  Since $J$ and $I_1$ each contain an open subset of $\mathring{K}$, there
  exists $\gamma \in \Gamma_T$ such that $\phi(\gamma)(S^1 \smallsetminus J) \subset I_1$,
  hence $\phi(\gamma b\gamma^{-1})$ has support inside $I_1$.
  Now take $U \subset I_1$ to be a connected component of $S^1\smallsetminus K$.
  Since the action of $BS(1,n)$ on $I_1$ is standard, there is some
  $w\in BS(1,n)$  mapping the support of $\phi(\gamma b\gamma^{-1})$ into $U$, and so
  \[ \Supp(\phi(w\gamma b\gamma^{-1}w^{-1})) \subset U.\]
  From this we will derive a contradiction with the relations satisfied by $\Gamma$.
  
  Let $g=w\gamma b\gamma^{-1}w^{-1}\in\Diff_+^\infty(S^1)$.
  Let $x\in J\cap\mathring{K}$ be a point with $s(x)\not\in\Supp(b)$, and
  take $y\in\mathring{K}$ such that $s(y) \notin \Fix(g)$.
  Then, for any small neighborhoods $U_x$ of $s(x)$ and $U_y$ of $s(y)$, if
  $\nu\in\Gamma_T$ satisfies
  $\nu(S^1\smallsetminus U_x) \subset  U_y$, the maps $\nu g\nu^{-1}$ and $g$
  do not commute.
  However, the maps $\phi(\nu g\nu^{-1})$ do commute, for they have disjoint
  supports, a contradiction.  We conclude that $\phi(\Gamma_T)$ acts minimally on~$S^1$.
  
  Since the action of $\phi(\Gamma_T)$ is minimal, it is topologically conjugate
  to the standard action of $\Gamma_T$.
  Thus, after conjugation by some $f \in \Homeo(S^1)$, we may assume that
  $\phi$ restricts to the identity morphism on $\Gamma_T$.
  Now $\Gamma_T$ has the contracting property, so by Theorem~\ref{theo:RecoGene} it
  recognizes $a$, $b$ and $\mu$, which have non-total support.
  It follows that $\phi$ is obtained by conjugation by $f$.
  Finally, Lemma~\ref{prop:Takens} asserts that the map $f$
  is $C^{1+\alpha}$ on the interval $I$ where $a, b$ and $\mu$ are supported.
  By minimality of the action of $\Gamma$, we conclude that
  $f$ is $C^{1+\alpha}$ everywhere, and the Theorem is proved.
\end{proof}

As a consequence, we have the following.
\begin{proof}[Proof of Theorem \ref{coro:CritReg} for $M=S^1$]
Since the map $\Gamma$ constructed above acts on $S^1$ with the contraction
property, and contains maps with non total support, $\Gamma$ also contains
maps with small supports everywhere. By Theorem~\ref{theo:RecoGene}, it thus
follows that $\Gamma$ recognizes all of $\Homeo(S^1)$.
Proposition~\ref{prop:DRCR} now proves Theorem~\ref{coro:CritReg} in the case $M=S^1$.
\end{proof}


\section{Rigidity and critical regularity for actions on $\R$ and $[0,1]$}\label{sec:CritRegInterval}
We proceed with the proofs of Proposition~\ref{prop:Rigid_lifts} and 
Theorem~\ref{thm:IntervalRigid}.  Recall that, as noted in the introduction,
we will be forced to work with slight modifications of the notion of map
recognition rather than directly applying Proposition~\ref{prop:DRCR}.

\subsection{Groups acting on the line}
We have an exact sequence
\[ 1\to \Z \to \HOZ \to \HOS \to 1, \]
where $\HOZ$ is the group of all homeomorphisms of the real line which
commute with the map $z\colon x\mapsto x+1$.
Let $\Gamma$ be the group from the previous section,
and let $\widetilde{\Gamma}$ denote its preimage in
$\HOZ$.
Thus, the group $\widetilde{\Gamma}$ is a subgroup of $\Diff_+^\infty(\R)$,
and it is not hard to check that it is also generated by $6$ elements.
\begin{proposition} \label{prop:Gamma0Rigid}
  The group $\widetilde{\Gamma}$ is $C^{1+\alpha}$-rigid.
\end{proposition}
\begin{proof}
  Much of this proof is an adaptation of an argument by Calegari,
  see~\cite{CalegariForcing}.
  
  Consider a faithful morphism $\phi\colon\widetilde{\Gamma}\to\Diff_+^{1+\alpha}(\R)$,
  for some $\alpha>0$.
  The elements $s,r,t\in\Gamma_T\subset\Gamma$ admit lifts
  $\widetilde{s}$, $\widetilde{r}$ and $\widetilde{t}$ satisfying
  $\widetilde{s}^2=\widetilde{r}^3=\widetilde{t}^7=z$.
  Suppose that $\phi(z)$ admits a fixed point in $\R$. Then
  $\phi(\widetilde{s})$, $\phi(\widetilde{r})$ and $\phi(\widetilde{t})$
  each fixes pointwise the fixed point set of $\phi(z)$, simply because
  the dynamics of any map on $\R$ is monotone on its orbits.
  Hence $\widetilde{\Gamma_T}$ has a global fixed point in $\R$, but this
  violates the Thurston stability theorem of \cite{ThurstonStability}.
  
  Thus $\phi(z)$ has no fixed point in $\R$, and so $\phi(z)$ is
  topologically conjugate to the map $z\colon x\mapsto x+1$ itself.
  As $\phi(z)$ is central in $\phi(\widetilde{\Gamma})$, the map $\phi$
  descends to the quotient, defining a faithful morphism
  $\overline{\phi}\colon\Gamma\to\Diff_+^{1+\alpha}(S^1)$, which is
  a $C^{1+\alpha}$-conjugation by Theorem~\ref{thm:GammaRigid}
  above. The conjugating map then lifts to a $C^{1+\alpha}$ diffeomorphism of $\R$,
  realizing $\phi$ by conjugation.
\end{proof}

Using this, we complete the proof of Proposition \ref{prop:Rigid_lifts}, showing
that $\widetilde{\Gamma}$ recognizes maps in $\HOZ$ up to integer translation.
\begin{proof}[Proof of Proposition \ref{prop:Rigid_lifts}]
We have already shown that $\widetilde{\Gamma}$ is $C^{1+\alpha}$-rigid.  
  So let $h\in\HOZ$ and let $f\in\Homeo_+(\R)$ be any map, and suppose that
  there is a group isomorphism
  $\phi\colon\langle\widetilde{\Gamma},h\rangle\to\langle\widetilde{\Gamma},f\rangle$
  which restricts to the identity on $\widetilde{\Gamma}$ and with
  $\phi(h)=f$.  We want to prove that $f=h \circ z^k$ for some $k$. 
  
  First, as $h$ commutes with $z$, so does $f$, and $f\in\HOZ$.
  Hence, $f$ descends to a homeomorphism $\overline{f}$ of the circle
  $\R/\Z$. Since the cyclic group generated by $z$ is central in both
  $\langle\widetilde{\Gamma},h\rangle$ and $\langle\widetilde{\Gamma},h\rangle$,
  the map $\phi$ descends to a group isomorphism between the quotients
  $\langle\widetilde{\Gamma},h\rangle/\langle z\rangle$ and
  $\langle\widetilde{\Gamma},f\rangle/\langle z\rangle$.
  These groups are naturally isomorphic to
  $\langle\Gamma,\overline{h}\rangle$ and $\langle\Gamma,\overline{f}\rangle$, where
  $\overline{h}$ is the homeomorphism of the circle defined by $h$.
  Now, the group $\Gamma$ has maps with small supports everywhere so 
  it follows from Theorem~\ref{theo:RecoGene} that the maps
  $\overline{h}$ and $\overline{f}$ agree.
  Hence, $h$ and $f$ may differ only by an integer translation.
\end{proof}
Combining Propositions~\ref{prop:DRCR} and~\ref{prop:Rigid_lifts} gives
the following, which proves Theorem~\ref{coro:CritReg} for~$M=\R$.
\begin{corollary}[Critical regularity on the line]
  Let $f \in \HOS$ be a map with non-total support, and suppose
  $f \notin \Diff^\beta(S^1)$ for some $\beta >1$.
  Let $\tilde{f}$ be a lift of $f$.
  Then $\langle \widetilde{\Gamma}, f \rangle$ is not isomorphic to
  any subgroup of $\Diff^\beta(\R)$.
\end{corollary}


\subsection{The closed interval}
We recall the set-up of Theorem~\ref{thm:IntervalRigid}.
Fix $1<n \in \N$ and consider the affine group generated by $BS(1,n)$ and
an irrational dilatation $\mu$, as in Section~\ref{sec:DRCR_cercle}.
Let $f \in \Homeo(\R)$ have compact support.
For simplicity we will also suppose that $\{ x\in\R\mid f(x)\neq x\}$
is an interval (this assumption is not strictly necessary, but it will make our
argument somewhat shorter).
Let $\Gamma$ denote the group generated by  $BS(1,n)$, $\mu$, and $f$, and
suppose $\phi\colon\Gamma\to\Diff^{1+\alpha}([0,1])$ is a faithful morphism.
We will show that there exists an interval
$(a,b)\subset[0,1]$, invariant under $\phi(\Gamma)$, and a
$C^{1+\alpha}$-diffeomorphism $h\colon\R\to(a,b)$ conjugating 
$\phi(\Gamma)$ with the standard action on $\R$. 

Note that this will also immediately imply the remaining case of the critical
regularity statement given in Theorem~\ref{coro:CritReg} in the introduction.

\begin{proof}[Proof of Theorem \ref{thm:IntervalRigid}]
  Let $\phi\colon\Gamma\to\Diff^{1+\alpha}([0,1])$ be as above. 
  Corollary~\ref{cor:BMNR} states that the complement of $\Fix(\phi(b))$ is a union of
  disjoint intervals $I_1=(a_1,b_1)$,\ldots,$I_m=(a_m,b_m)$, each of which admits
  a homeomorphism
  $\psi_j\colon\R\to I_j$ conjugating the standard action of $BS(1,n)$ on $\R$
  with its action \textsl{via} $\phi$ on $I_j$.
  The proof has three steps, which we separate into short lemmas.
   
  \begin{lemma} 
  $\phi(f)$ preserves each interval $I_j$.
  \end{lemma} 
  
  \begin{proof} 
  For this, it suffices to show that $\phi(f)(a_1)=a_1$,
  and $\phi(f)(b_1)=b_1$, as the remaining intervals can then be shown 
  invariant by applying this argument iteratively to the restriction of the
  action to $[b_1, 1]$ and so on.
 
  So, suppose for contradiction
  that $\phi(f)(a_1) \neq a_1$, up to replacing $f$ with its inverse we may assume that 
  $\phi(f)(a_1)<a_1$. Then we also have $\phi(f)(a_1+\varepsilon)<a_1$ for
  some $\varepsilon>0$. Let $x_0=f(a_1+\varepsilon)$.
  Then, for all $N\in\Z$, we have
  $\phi(b^Nf^{-1}b^{-N})\circ\phi(f)(x_0)=x_0$,
  while $\phi(f)\circ\phi(b^Nf^{-1}b^{-N})(x_0)\neq x_0$ for all $N$.
  This contradicts that $f$ and $b^Nfb^{-N}$ commute for $N$ large
  enough, hence $\phi(f)(a_1)=a_1$. 
  
  Now suppose for contradiction that $\phi(f)(b_1)>b_1$.
  Let $x_0=\phi(f)^{-1}(b_1)\in(a_1,b_1)$.
  As before let $\tau_k \in BS(1,n)$ denote translation by $k$.
  For all $k$ large, $\tau_k f\tau_k^{-1}$
  and $f$ commute, hence
  $\phi(\tau_k f\tau_k^{-1})\circ\phi(f)(x_0)=\phi(f)\circ\phi(\tau_kf\tau_k^{-1})(x_0)$.
  Since the set $\phi(\tau^{-1}_k)(x_0)$ accumulates
  to $a_1$, we get that $\phi(f)(x)>x$ for a dense (and open)
  set of points near $a_1$.  

  We claim next that $\Fix(\phi(f)) \cap (a,b]= \emptyset$, \textsl{i.e.}
  $\phi(f)$ is strictly increasing on this interval.
  Indeed, if $c \in (a,b]$ is a least fixed point of $\phi(f)$, then
  $\phi(\tau_k f\tau_k^{-1})$ fixes $\phi(\tau_k)(c)$ and is increasing on $(a,c)$.
  Since this map commutes with $\phi(f)$, provided $k$ is large, $\phi(f)$ also
  fixes $\phi(\tau_k)(c)$.
  It follows that $\Fix(\phi(f))$ accumulates at $b_1$, so $b_1$ is actually
  fixed by $\phi(f)$, contradicting our assumption.

  Note that this argument applies not only to $f$ but to any compactly supported
  homeomorphism $g \in \Gamma$ with $\phi(g)(b_1) > b_1$. In particular we
  may fix $k$ large and take $g = \tau^{-1}_k f \tau_k f^{-1}$ and conclude
  $\phi(\tau^{-1}_k f \tau_k f^{-1})(x) > x$ for all $x \in (a_1, b_1)$ or
  equivalently that $\phi(\tau^{-1}_k f \tau_k)(y) > \phi(f)(y) > y$ for
  all $y \in (a_1, x_0)$.   

  Now we adapt the proof of Lemma \ref{lem:PtFixe} to derive a contradiction.
  The proof of the Lemma said that if $u$ and $v$ are commuting maps of $\R$,
  or equivalently, commuting maps of an interval $(a,b)$ with $id < u < v$,
  and $w$ satisfies $wuw^{-1} = v$ and $w(x) > x$ for some $x \in (a,b)$,
  then we can find a point $y \in (a, x)$ with $w(y)<y$.
  We assumed there that $u, v$ and $w$ preserved the interval $(a,b)$.
  However, the exact same argument applies to our situation using the maps
  $u = \phi(f)$, $v =\phi(\tau^{-1}_k f \tau_k f^{-1})$, and $w = \phi(f)$,
  all of which fix the point $a_1$.
  Choose a point $x \in (a_1, x_0)$, so we know already that $\phi(f)(x) > x$.
  Running the proof of the lemma (verbatim) shows that there is some point
  $y \in (a_1, x)$ with $\phi(f)(y) < y$, contradicting our first observation above.
  \end{proof} 
  
  \begin{lemma} \label{lem:Conj1}
  The commutator $\phi([f,b])$ acts nontrivially on some $I_j$, and on any such
  interval, $\psi_j$ conjugates the action of $\langle BS(1,n), f\rangle$
  to the standard action.
  \end{lemma} 
  \begin{proof}
  Since $\phi$ is faithful, the commutator $[f,b]$ acts nontrivially and so
  $f$ acts nontrivially on at least one of the ($f$-invariant) intervals $I_j$
  in the complement of $b$.
  Identifying $I_j$ with $\R$ \textsl{via} $\psi_j$, we obtain a morphism from
  $\langle BS(1,n), f \rangle$ to $\langle BS(1,n),  \phi(f) \rangle$
  that is the identity on $BS(1,n)$ and sends $f$ to $\phi(f)$.
  By Proposition~\ref{prop:BSRecoMieux}, we conclude that, on any such interval,
  $\phi(f)$ either acts as a translation (which does not occur if
  $\phi([f,b]) \neq \id$ on $I_j$), or we have
  $\phi(f)(x) \in \{x, \psi_j f(x) \psi_j^{-1}\}$ for all $x \in I_j$.
  In this latter case, it follows that the interior of the support of
  $\phi(f)$ in $I_j$ is a union of connected components of the interior of
  the support of $\psi_j f\psi_j^{-1}$.
  But $\Supp(f)$ was assumed connected, so we have proved the lemma.
\end{proof}  

\begin{lemma}\label{lem:ConjGamma}
  For some $j$, the map $\psi_{j}$ conjugates the action of $\Gamma$ on $\R$
  to that of $\phi(\Gamma)$ on $I_j$.
\end{lemma}
  
\begin{proof}
  Let $g = \mu f \mu^{-1}$.
  Applying Lemma \ref{lem:Conj1} to $g$ in place of $f$ shows that whenever
  $\phi([g,b])$ acts nontrivially on some $I_j$, then $\psi_j$ conjugates
  the action of $\langle BS(1,n), g \rangle$ on $I_j$ to the standard action.
  Since $\phi$ is faithful, there is some interval $I_j$ where $\phi([f,b])$
  and $\phi([g,b])$ are simultaneously nontrivial.
  On this interval, we will easily be able to show that $\psi_j$ conjugates
  $\mu$ to the standard action as well, and hence conjugates all of $\Gamma$.

  To see this, for any $x \in I_j$, take a nested sequence $U_{k, x}$ of
  intervals with $\bigcap_k U_{k,x} = \{x\}$.
  Since the action of $\langle BS(1,n), f \rangle$ on $I_j$ has small supports
  everywhere, we may take $\gamma_k \in \langle BS(1,n), f \rangle$ with
  $\phi(\gamma_k)$ supported on $U_{k, x}$.
  Thus, $\gamma_k$ is supported on $\psi_j^{-1}(U_{k, x}$, and its conjugate
  by $\mu$ is supported on $\mu (\psi_j^{-1}U_{k, x})$.
  Applying Proposition~\ref{prop:BSRecoMieux} to
  $\langle BS(1,n), \mu \gamma_k \mu^{-1} \rangle$, it follows that
  $\phi(\mu \gamma_k \mu^{-1})$ is supported on
  $\psi_j (\mu) \psi_j^{-1} (U_{k, x})$.
  We conclude that, $\phi(\mu)(x) = \psi_j \mu \psi_j^{-1}(x)$, as desired.
  \end{proof} 

 \smallskip
 \noindent {\bf Conclusion of proof.} 
 It remains only to remark that, by Lemma~\ref{prop:Takens}, the map 
 $\psi_{j}$ obtained from Lemma~\ref{lem:ConjGamma} is of class $C^{1+\alpha}$.
 \end{proof}


\bibliographystyle{plain}

\bibliography{biblio_new}

\end{document}